\documentclass[a4paper,10pt,reqno]{amsart}
\usepackage[utf8]{inputenc}
\usepackage{amsmath,amsfonts,amssymb,mathrsfs,latexsym,bbm,enumitem,xcolor}
\usepackage[all,cmtip]{xy}
\newtheorem{definition}{Definition}
\newtheorem{theorem}[definition]{Theorem}
\newtheorem{lemma}[definition]{Lemma}
\newtheorem{coro}[definition]{Corollary}
\newtheorem{proposition}[definition]{Proposition}
\newtheorem{remark}[definition]{Remark}
\begin{document}
\title{The equivariant Milnor--Witt motive of \(\overline{\mathcal{M}}_{1,2}\)}
\author{Nanjun Yang}
\address{Nanjun Yang\\Yanqi Lake Beijing Institute of Mathematical Sciences and Applications\\Huai Rou District\\Beijing China}
\email{ynj.t.g@126.com}
\begin{abstract}
We provide a decomposition of the equivariant Milnor-Witt motive for the moduli space of stable curves $\overline{\mathcal{M}}_{1,2}$. As a result, the equivariant Chow-Witt ring $\widetilde{CH}^*(\overline{\mathcal{M}}_{1,2})$ is fully determined.
\end{abstract}
\maketitle

\section{Introduction}

The moduli spaces of stable curves, $\overline{\mathcal{M}}_{g,n}$, are central objects of study in algebraic geometry, acting as classifying spaces that bridge topology, geometry, and mathematical physics. A fundamental goal is to understand their algebraic cycles and motives. In genus zero, the geometry is remarkably well-behaved: $\overline{\mathcal{M}}_{0,n}$ is completely cellular. Consequently, its classical integral Voevodsky motive decomposes strictly as a direct sum of shifted Tate motives, $\mathbb{Z}(i)[2i]$. However, while this classical motivic framework is powerful, its Betti realization only captures the singular cohomology of the complex points of the moduli space.

To access finer topological data—specifically, the singular cohomology of real points—this rigid classical structure must be upgraded to include "quadratic" information. Recently, this was successfully achieved by lifting computations into the realm of the Chow-Witt ring (\cite{F1}). Using the Milnor–Witt (MW) motive developed by Calmès, Déglise, and Fasel (\cite{BCDFO}), the classical category is enriched by the Grothendieck-Witt ring of quadratic forms. In this refined setting, the Milnor–Witt motive of $\overline{\mathcal{M}}_{0,n}$ expands to a direct sum of $\mathbb{Z}(i)[2i]$ and $\mathbb{Z}/\eta(i)[2i]$ over $\mathbb{R}$ (\cite{FY},\cite{EHKR}), capturing topological nuances invisible to standard Voevodsky motives.

Extending this quadratic enrichment from genus zero to genus one presents profound structural and homological challenges. The primary obstacles are the presence of stacky torsion (due to curves with automorphisms) and intricate boundary singularities. For classical Chow motives with rational coefficients, the limitations are well understood: the rational motive of $\overline{\mathcal{M}}_{1,n}$ is of pure Tate type for $n \leq 10$, but non-vanishing odd cohomology breaks this property for $n > 10$. With integral coefficients, remarkable progress has been made for small $n$, such as Bishop's analysis of the integral Chow ring of $\mathcal{M}_{1,n}$ for $n \leq 10$ (\cite{B}), and Newman's proof that the integral motive with compact support of $\overline{\mathcal{M}}_{1,n}$ is mixed Tate for $n \leq 4$ (\cite{N}).

These classical integral computations fundamentally rely on exhaustive characteristic class tracking and boundary stratifications. However, adapting these algebraic stratification techniques to the Milnor–Witt category introduces structural complications. Because the Hopf element $\eta$ does not vanish, explicit block-by-block splittings inevitably tangle the motive into a complex web of mapping cones. In this paper, we do utilize such an explicit mapping-cone decomposition of $\overline{\mathcal{M}}_{1,1}$ (Theorem \ref{11}) to deduce the explicit additive group structure of $\overline{\mathcal{M}}_{1,2}$. However, this formulation obscures the multiplicative structure. To compute the full Chow-Witt ring structure, we cannot rely on this explicit mapping-cone breakdown; a cleaner, intrinsic approach is required.

Denote by  $\mathrm{inf}:\overline{\mathcal{M}}_{1,1}\to\overline{\mathcal{M}}_{1,2}$ the infinity section. Our main theorem is the following (Theorem \ref{12}):

\begin{theorem}
The composite
$$\mathbb{Z}(\overline{\mathcal{M}}_{1,2}\setminus \mathrm{inf}) \to \mathbb{Z}(\overline{\mathcal{M}}_{1,2}) \to \mathbb{Z}(\overline{\mathcal{M}}_{1,1})$$
is an isomorphism in $\widetilde{DM}(k)$
and
$$\mathbb{Z}(\overline{\mathcal{M}}_{1,2}) \cong \mathbb{Z}(\overline{\mathcal{M}}_{1,1}) \oplus Th_{\overline{\mathcal{M}}_{1,1}}(O(-1))$$
as Milnor–Witt motives, where $O(1)$ is the Hodge bundle generating $\mathrm{Pic}(\overline{\mathcal{M}}_{1,1}) \cong \mathbb{Z}$ and $Th_{\overline{\mathcal{M}}_{1,1}}(O(-1))$ is the Thom space of $O(-1)$.
\end{theorem}

Combining the theorem above and the presentation of the total Chow-Witt ring of $\overline{\mathcal{M}}_{1,1}$ recently established in \cite[Theorem 5.5.3]{LM}, we determine the Chow-Witt ring $\widetilde{CH}^*(\overline{\mathcal{M}}_{1,2})$ (Theorem \ref{ringex}).

Denote by $GW(k)$ the Grothendieck-Witt ring of quadratic forms on $k$, by $h\in GW(k)$ the hyperbolic element and by $I(k)$ its fundamental ideal.

\begin{theorem}\label{ringex}
The $\widetilde{CH}^*(\overline{\mathcal{M}}_{1,2})$ is isomorphic to
$$\frac{GW(k)[TH,V,T^2,T',H']}{\begin{pmatrix}I(k)T',I(k)H',I(k)TH,I(k)T^2,V^2,(TH)^2-2hT^2,hV,24T^2,12T^2TH-VT^2,THV,\\H'V,T^2H'-THT',H'^2-2hT',H'^2-THH',T'^2-T^2T',T'H'-T^2H',12T^2H'-VT'\end{pmatrix}}$$
as $\mathbb{N}$-graded algebras, where $\mathrm{deg}(TH)=\mathrm{deg}(V)=\mathrm{deg}(H')=1$ and $\mathrm{deg}(T^2)=\mathrm{deg}(T')=2$.
\end{theorem}
\textbf{Organization of the paper. }
Section \ref{eqmw} recalls the basic definitions and foundational constructions of equivariant Milnor--Witt motives, providing the categorical framework used throughout the work. 
Section \ref{m11} is devoted to the computation of the MW-motive of the moduli stack $\overline{\mathcal{M}}_{1,1}$. 
Given the decomposition of $\overline{\mathcal{M}}_{1,1}$ obtained in Theorem \ref{11} into the distinguished blocks, 
Section \ref{block} determines their respective MW-motivic cohomologies. Finally, Section \ref{m12} computes the MW-motive and the Chow-Witt ring of $\overline{\mathcal{M}}_{1,2}$.

\textbf{Conventions.}
Throughout, denote by $Sm/k$ the category of smooth separated schemes over a field $k$ with $\operatorname{char}(k)=0$ and by \(C(f)\) the mapping cone of a morphism \(f\). 

For any $X \in Sm/k$, we write $H_{MW}^{\ast,\ast}(X,\mathbb{Z})$ and $H_{M}^{\ast,\ast}(X,\mathbb{Z})$ for the Milnor--Witt motivic cohomology (\cite{BCDFO}) and motivic cohomology (\cite{MVW}) of $X$, respectively. In bidegree $(2n,n)$, these groups specialize to the Chow--Witt group $\widetilde{CH}^n(X)$ and the Chow group $CH^n(X)$, respectively. All \(Hom\)-groups are taken in the category \(\widetilde{DM}(k)\) of Milnor--Witt motives.

\textbf{Acknowledgement. }The author was partially supported by the National Natural Science Foundation of China (Grant No. 12201336). He would like to thank Jin Cao for careful reading of the text, as well as helpful discussions and suggestions.

\section{Equivariant Milnor--Witt motives}\label{eqmw}


In this section, we introduce the basic definition of equivariant Milnor--Witt motives, following the ideas of \cite{EG}.

Let \(G\) be a linear algebraic group and \(X\in Sm/k\) be a \(G\)-scheme. Suppose that for each \(n\in\mathbb{N}\), there is a linear \(G\)-representation \(V_n\) with an open set \(U_n\subseteq V_n\) such that
\begin{enumerate}
\item The \(U_n\) is \(G\)-stable and its \(G\)-action is free;
\item The \(X\times_GU_n\) is a smooth scheme;
\item The codimension of \(V_n\setminus U_n\) in \(V_n\) is \(\geq n\);
\item There is an equivariant map \(V_n\to V_{n+1}\) which maps \(U_n\) to \(U_{n+1}\).
\end{enumerate}

\begin{remark}
In practice, we can fix some \(V=V_1\) and define \(V_n=V^{\times n}\) by diagonal action. Define \(U_n=V^{\times n}\setminus(V\setminus U)^{\times n}\). Then we have equivariant maps \(\begin{array}{ccc}U_n&\to&U_{n+1}\\(x_1,\cdots,x_n)&\mapsto&(x_1,\cdots,x_n,0)\end{array}\).
\end{remark}

%
For any \(L\in Pic(X)\) and \(X\in Sm/k\), denote by \(Th(L)=\mathbb{Z}(L)/\mathbb{Z}(L^{\times})\) the Thom space.
\begin{definition}
For any \(X\in Sm/k\), \(L\in Pic(X)\), define the MW-motivic cohomology twisted by \(L\)
\[H^{p,q}_{MW}(X,\mathbb{Z},L)=H^{p+2,q+1}_{MW}(Th(L),\mathbb{Z}).\]
It only depends on the class \(L\in Pic(X)/2\).
\end{definition}
\begin{proposition}\label{appro}
The MW-motive \(\varinjlim_n\mathbb{Z}(X\times_GU_n)\) is independent of the choice of \((V_n,U_n)\).
\end{proposition}
\begin{proof}
Suppose \(U,Y\in Sm/k\), \(L\in Pic(U)\) and \(U\) is projective. We have
\begin{align*}
	&Hom(\mathbb{Z}(Y),Th(L)(i)[j])\\
=	&Hom(\mathbb{Z}(Y)\otimes Th(L\otimes\omega_U),\mathbb{Z}(i+dim(U)+1)[j+2dim(U)+2])\\
= 	&H^{j+2dim(U),i+dim(U)}_{MW}(U\times Y,\mathbb{Z},(L\otimes\omega_U)|_{U\times Y})
\end{align*}
by \cite[Theorem 6.1]{Y}, which vanishes if \(j-i>dim(Y)\).

If \(U\) is not projective, by resolution of singularities we find \(U\subseteq\bar{U}\) such that the \(\bar{U}\) is smooth projective and \(Z=\bar{U}\setminus U\) is a simple normal crossing divisor \(\cup_{i=1}^mZ_i\). Suppose \(L=L'|_U\) where \(L'\in Pic(\bar{U})\). We prove by induction on \(m\) that
\begin{equation}\label{vanish}Hom(\mathbb{Z}(Y),Th(L)(i)[j])=0, j-i>dim(Y).\end{equation}
The \(m=0\) was established so suppose the \eqref{vanish} is true for \(0\leq m<m_0\). Denote by \(W_i=\bar{U}\setminus(\cup_{j=1}^iZ_i)\). There is an exact sequence
\[\xymatrix{Hom(\mathbb{Z}(Y),Th(L'|_{Z_{m_0}\cap W_{m_0-1}}\otimes det(N_{Z_{m_0}/\bar{U}}))(i+1)[j+1])\ar[r]&Hom(\mathbb{Z}(Y),Th(L'|_{W_{m_0}})(i)[j])\ar[d]\\& Hom(\mathbb{Z}(Y),Th(L'|_{W_{m_0-1}})(i)[j])},\]
hence the \eqref{vanish} holds for \(m=m_0\) if \(j-i>dim(Y)\) by induction hypothesis, completing the induction.

Let us suppose that we have an open immersion \(U\subseteq V\), \(L=L'|_U\) and \(Z=V\setminus U\) has codimension \(c\). We can take the smooth locus \(W\) of \(Z\) hence we obtain a smaller subset \(Z\setminus W\). Iterating we find a filtration of open subsets
\[U=U_1\subseteq U_2\subseteq\cdots\subseteq U_m=V\]
such that each \(Z_s=U_{s+1}\setminus U_s\) is smooth with codimension \(c_s\geq c\). We have an exact sequence
\[\cdots Hom(\mathbb{Z}(Y),Th(L'|_{U_s})(i)[j])\to Hom(\mathbb{Z}(Y),Th(L'|_{U_{s+1}})(i)[j])\to Hom(\mathbb{Z}(Y),Th(M)(i+c_s)[j+2c_s])\]
where \(M=L'|_{Z_s}\otimes N_{Z_s/U_{s+1}}\). So by \eqref{vanish}, we have
\[Hom(\mathbb{Z}(Y),Th(L'|_{U_s})(i)[j])=Hom(\mathbb{Z}(Y),Th(L'|_{U_{s+1}})(i)[j])\]
if \(j+c_s-i>dim(Y)+1\). So if \(j+c-i>dim(Y)+1\) we have
\[Hom(\mathbb{Z}(Y),Th(L)(i)[j])=Hom(\mathbb{Z}(Y),Th(L')(i)[j]).\]

Let us take \((V_n,U_n)\) and \(V'_n,U'_n\) which are approximations of \([X/G]\). Define \(W_n=(U_n\times V'_n)\cup(V_n\times U'_n)\). Then \(X\times_G(V_n\times U'_n)\subseteq X\times_GW\) with complement being \(X\times_G(U_n\times(V'_n\setminus U'_n))\). So we have
\[Hom(\mathbb{Z}(Y),\mathbb{Z}(X\times_G(V_n\times U'_n))(i)[j])=Hom(\mathbb{Z}(Y),\mathbb{Z}(X\times_GW_n)(i)[j])\]
if \(j+n-i>dim(Y)+1\). By \(\mathbb{A}^1\)-invariance the first term is equal to
\[Hom(\mathbb{Z}(Y),\mathbb{Z}(X\times_GU'_n)(i)[j]).\]
Let \(n\to\infty\) we obtain
\[Hom(\mathbb{Z}(Y),\varinjlim_n\mathbb{Z}(X\times_GU'_n)(i)[j])=Hom(\mathbb{Z}(Y),\varinjlim_n\mathbb{Z}(X\times_GW_n)(i)[j])\]
for all \(i,j\). By similar statement, we obtain
\[Hom(\mathbb{Z}(Y),\varinjlim_n\mathbb{Z}(X\times_GU_n)(i)[j])=Hom(\mathbb{Z}(Y),\varinjlim_n\mathbb{Z}(X\times_GW_n)(i)[j]).\]
Since \(\widetilde{DM}(k)\) is compactly generated, we see that
\[\varinjlim_n\mathbb{Z}(X\times_GU_n)=\varinjlim_n\mathbb{Z}(X\times_GU'_n)\]
as MW-motives.
\end{proof}
\begin{definition}
We define the equivariant MW-motive for $[X/G]$ to be:
\[\mathbb{Z}([X/G])=\varinjlim_n\mathbb{Z}(X\times_GU_n).\]

Its equivariant MW-motivic cohomologies are
\[H^{*,*}_{MW}([X/G],\mathbb{Z})=\varprojlim_nH^{*,*}_{MW}(X\times_GU_n,\mathbb{Z}).\]
\end{definition}
\begin{proposition}
If \([Y/H]\cong[X/G]\) as stacks where \(Y\in Sm/k\) and \(H\) is a linear algebraic group, satisfying (1)-(4) in the beginning of this section, we have \(\mathbb{Z}([X/G])\cong\mathbb{Z}([Y/H])\).
\end{proposition}
\begin{proof}
Let \(Z=X\times_{[X/G]}Y\) via the isomorphism \([Y/H]\cong[X/G]\) and let \((V_n,U_n),(V'_m,U'_m)\) be approximation of \(G,H\), respectively. The \(Z\) is a principal \(G\)-bundle (resp. principal \(H\)-bundle) over \(X\) (resp. \(Y\)). For any \(W\in Sm/k\), by similar discussion as in Proposition \ref{appro}, we have
\begin{align*}
	&\varinjlim_{m\to\infty}Hom(\mathbb{Z}(W),\mathbb{Z}(Z\times_{G\times H}(U_n\times U'_m))(i)[j])\\
=	&\varinjlim_{m\to\infty}Hom(\mathbb{Z}(W),\mathbb{Z}(Z\times_{G\times H}(U_n\times V'_m))(i)[j])\\
=	&Hom(\mathbb{Z}(W),\mathbb{Z}(X\times_GU_n)(i)[j]).
\end{align*}
So we have
\[\varinjlim_{n,m\to\infty}Hom(\mathbb{Z}(W),\mathbb{Z}(Z\times_{G\times H}(U_n\times U'_m))(i)[j])=Hom(\mathbb{Z}(W),\mathbb{Z}([X/G])(i)[j]).\]
Similarly, we have
\[\varinjlim_{n,m\to\infty}Hom(\mathbb{Z}(W),\mathbb{Z}(Z\times_{G\times H}(U_n\times U'_m))(i)[j])=Hom(\mathbb{Z}(W),\mathbb{Z}([Y/H])(i)[j]).\]
Hence the statement follows.
\end{proof}
\begin{remark}
Any \(L\in CH^1([X/G])\) is given by \(L_n\in Pic(X\times_GU_n)\) such that \(L_n|_{X\times_GU_{n-1}}=L_{n-1}\). Hence by the same argument one can define
\[Th(L)=\varinjlim_nTh(L_n).\]
\end{remark}

Let us establish the following notations for computing MW-motivic cohomologies with $\mathbb{Z}/\eta$-coefficients.
\begin{definition} \label{hopfele}
Denote by \(\eta:\mathbb{Z}(1)[1]\to\mathbb{Z}\) the Hopf element, by \(\mathbb{Z}/\eta\) the mapping cone of \(\eta\) and by \(\partial:\mathbb{Z}/\eta\to\mathbb{Z}(1)[2]\) the boundary map.
\end{definition}
\begin{definition}
Define
\[E_1^{p,q}(A)=Hom_{\widetilde{DM}(k)}(A,\mathbb{Z}/\eta(q)[p])\]
for every MW-motive \(A\).
\end{definition}
The \(E_1^{*,*}(A)\) satisfies the long exact sequence
\[\cdots\to H^{p+1,q+1}_{MW}(A,\mathbb{Z})\xrightarrow{\eta}H^{p,q}_{MW}(A,\mathbb{Z})\to E_1^{p,q}(A)\xrightarrow{\partial}H^{p+2,q+1}_{MW}(A,\mathbb{Z})\to\cdots.\]
\begin{remark}\label{e1}
By \cite[Theorem 4.13]{Y1}, if \(CH^*(X)\) is \(2\)-torsion free for \(X\in Sm/k\), there is a Cartesian square
\[
\xymatrix@C=3.5pc{
    E_1^{2n,n}(X) \ar[r] \ar[d] & CH^{n+1}(X) \ar[d] \\
    CH^n(X) \ar[r]^-{Sq^2 \circ \bmod 2} & CH^{n+1}(X)/2
}.
\]
\end{remark}

\section{The MW-motive of \(\overline{\mathcal{M}}_{1,1}\)}\label{m11}
We first recall the computation of the MW-motive of \(\mathbb{P}^{\infty}\).

\begin{proposition}\label{Pn}
We have
\[\mathbb{Z}(\mathbb{P}^{\infty})\xrightarrow[\cong]{(p,c^{2i-1})}\mathbb{Z}\oplus\oplus_{i=1}^{\infty}\mathbb{Z}/\eta(2i-1)[4i-2],\]
where  \(c^{2i-1}\in{E}_1^{4i-2,2i-1}(\mathbb{P}^{\infty})\) is given by
\[(c_1(O(1))^{2i-1},c_1(O(1))^{2i})\in CH^{2i-1}(\mathbb{P}^{\infty})\oplus CH^{2i}(\mathbb{P}^{\infty}).\]
Similarly, we have
\[Th_{\mathbb{P}^{\infty}}(O(1))\xrightarrow[\cong]{d^{2i-1}}\oplus_{i=1}^{\infty}\mathbb{Z}/\eta(2i-1)[4i-2]\]
where \(d^{2i-1}\in{E}_1^{4i-2,2i-1}(Th_{\mathbb{P}^{\infty}}(O(1)))\) is given by
\[(c_1(O(1))^{2i-2},c_1(O(1))^{2i-1})\in CH^{2i-2}(\mathbb{P}^{\infty})\oplus CH^{2i-1}(\mathbb{P}^{\infty}).\]
\end{proposition}
\begin{proof}
Follows from \cite[Theorem 5.11]{Y1} and \cite[Theorem 4.19]{Y2}.
\end{proof}
\begin{lemma}\label{o1}
We have
\[\mathbb{Z}(O_{\mathbb{P}^{\infty}}(1)^{\times})=\mathbb{Z}\]
in \(\widetilde{DM}(k)\).
\end{lemma}
\begin{proof}
We have the Gysin triangle
\[\mathbb{Z}(O_{\mathbb{P}^{\infty}}(1)^{\times})\to\mathbb{Z}(\mathbb{P}^{\infty})\xrightarrow{c} Th_{\mathbb{P}^{\infty}}(O(1))\to\cdots[1]\]
where the \(c\) is exactly the map \((c^{2i-1})\). So we conclude.
\end{proof}

Next we compute the MW-motive of \(\overline{\mathcal{M}}_{1,1}\).

Given \(X\in Sm/k\) and \(f,g:\mathbb{Z}(X)\to C\) for \(C\in\widetilde{DM}(k)\), define \(f\boxtimes g:\mathbb{Z}(X)\to C\) by the composite
\[\mathbb{Z}(X)\xrightarrow{\triangle}\mathbb{Z}(X)\otimes\mathbb{Z}(X)\xrightarrow{f\otimes g}C.\]

Denote by \(V_{n_1,\cdots,n_s}=\mathbb{A}^s,n_i\in\mathbb{Z}\) with the \(\mathbb{G}_m\)-action \(t\cdot(x_1,\cdots,x_s)=(t^{n_1}x_1,\cdots,t^{n_s}x_s)\). Similarly we can define \(\mathbb{P}^{s-1}_{n_1,\cdots,n_s}\).  It is well known that:
\[\overline{\mathcal{M}}_{1,1}=[V_{4,6}^{\times}/\mathbb{G}_m]\]
as stacks.  See \cite[Proposition 4.2.3]{LM} for example.  Its equivariant approximation is \((O(4)\oplus O(6))^{\times}\) on \(\mathbb{P}^{\infty}\).

\begin{theorem}\label{11}
\begin{enumerate}
\item We have
\[\mathbb{Z}(\overline{\mathcal{M}}_{1,1})=\mathbb{Z}\oplus C(24\partial)(1)[1]\oplus\bigoplus_{i=1}^{\infty}C(24\mathrm{Id}_{\mathbb{Z}/\eta})(2i+1)[4i+1]\]
in \(\widetilde{DM}(k)\).
\item We have
\[Th_{\overline{\mathcal{M}}_{1,1}}(O(-1))=\mathbb{Z}/\eta(1)[2]\oplus\bigoplus_{i=1}^{\infty}C(24\mathrm{Id}_{\mathbb{Z}/\eta})(2i+1)[4i+1]\]
in \(\widetilde{DM}(k)\).
\end{enumerate}
\end{theorem}
\begin{proof}
\begin{enumerate}
\item We have a Gysin triangle
\begin{equation}\label{gysin}\mathbb{Z}((O(4)\oplus O(6))^{\times})\to\mathbb{Z}(\mathbb{P}^{\infty})\xrightarrow{e(O(4)\oplus O(6))}\mathbb{Z}(\mathbb{P}^{\infty})(2)[4]\to\cdots[1].\end{equation}

The composition
\[\mathbb{Z}(\mathbb{P}^{\infty})\xrightarrow{\mathrm{Id}_{\mathbb{P}^{\infty}}\boxtimes(24\partial\circ c^1)}\mathbb{Z}(\mathbb{P}^{\infty})(2)[4]\xrightarrow{c^{2i-1}(2)[4]}\mathbb{Z}/\eta(2i+1)[4i+2]\]
is equal \(24c^{2i+1}\) by Remark \ref{e1} and Proposition \ref{Pn}, whereas we have a commutative diagram
\[
	\xymatrix
	{
		\mathbb{Z}(\mathbb{P}^{\infty})\ar[r]^-{\mathrm{Id}\boxtimes 24\partial\circ c^1}\ar[d]_{24c^1}	&\mathbb{Z}(\mathbb{P}^{\infty})(2)[4]\ar[d]\\
		\mathbb{Z}/\eta(1)[2]\ar[r]^-{\partial(1)[2]}														&\mathbb{Z}(2)[4]
	}.
\]
So we have
\[e(O(4)\oplus O(6))=\mathrm{Id}_{\mathbb{P}^{\infty}}\boxtimes(24\partial\circ c^1).\]
So \((O(4)\oplus O(6))^{\times}[1]\) is the mapping cone of the map 
\[
	\xymatrixcolsep{0pc}
	\xymatrix
	{
		\mathbb{Z}\oplus	&\mathbb{Z}/\eta(1)[2]\oplus\ar[d]_{24\partial(1)[2]}	&\mathbb{Z}/\eta(3)[6]\oplus\ar[d]_{24\mathrm{Id}}	&\cdots\ar[d]_{24\mathrm{Id}}\\
									&\mathbb{Z}(2)[4]\oplus										&\mathbb{Z}/\eta(3)[6]\oplus						&\cdots
	},
\]
which concludes the proof.
\item We have a Gysin triangle
\[Th_{(O(4)\oplus O(6))^{\times}}(O(-1))\to Th_{\mathbb{P}^{\infty}}(O(-1))\xrightarrow{e(O(4)\oplus O(6))}Th_{\mathbb{P}^{\infty}}(O(-1))(2)[4]\to\cdots[1].\]
The \(Th_{(O(4)\oplus O(6))^{\times}}(O(-1))[1]\) is the mapping cone of the map (Remark \ref{e1}, Proposition \ref{Pn})
\[
	\xymatrixcolsep{0pc}
	\xymatrix
	{
		\mathbb{Z}/\eta(1)[2]\oplus		&\mathbb{Z}/\eta(3)[6]\oplus\ar[d]_{24\mathrm{Id}}	&\cdots\ar[d]_{24\mathrm{Id}}\\
													&\mathbb{Z}/\eta(3)[6]\oplus						&\cdots
	},
\]
which concludes the proof.
\end{enumerate}
\end{proof}

\section{Elaborating the Basic Blocks}\label{block}
Denote by \(GW(k)\) (resp. \(W(k)\)) the Grothendieck-Witt group the field \(k\) classifying nondegenerate (resp. anisotropic) quadratic forms over \(k\). The \(K_*^{M}(k)=\frac{T(k^{\times})}{<a\otimes(1-a)>},a\in k^{\times}\setminus1\) stands for Milnor K-theory of \(k\).

Given the decomposition of $\overline{\mathcal{M}}_{1,1}$ in Theorem \ref{11} into blocks, including $\mathbb{Z}$, $\mathbb{Z}/\eta$, $C(24\partial)$, and $C(24\mathrm{Id}_{\mathbb{Z}/\eta})$, we must determine their respective MW-motivic cohomologies. Notably, as demonstrated in \cite[Proposition 5.8]{Y1}, the motive $\mathbb{Z}/\eta$ possesses a strong dual: $\mathbb{Z}/\eta(-1)[-2]$.
\begin{lemma}\label{-1}
\begin{enumerate}
\item We have
\[Hom(\mathbb{Z},\mathbb{Z}/\eta(n)[2n-1])=\begin{cases}K_1^M(k)&n=1\\2K_1^M(k)&n=0\\0&n\neq0,1\end{cases}.\]
\item If \(n\) is even, we have
\[Hom(\mathbb{Z}/\eta,\mathbb{Z}/\eta(n)[2n-1])=\begin{cases}0&n\neq0,2\\K_1^M(k)&n=2\\2K_1^M(k)&n=0\end{cases}.\]
\end{enumerate}
\end{lemma}
\begin{proof}
\begin{enumerate}
\item We have
\[Hom(\mathbb{Z},\mathbb{Z}/\eta(n)[2n-1])=Hom(\mathbb{Z}/\eta,\mathbb{Z}(n+1)[2n+1])\]
by duality. Apply \cite[Proposition 2.2]{Y2}.
\item We have
\[Hom(\mathbb{Z}/\eta,\mathbb{Z}/\eta(n)[2n-1])=Hom(\mathbb{Z}/\eta,\mathbb{Z}(n-1)[2n-1])\oplus Hom(\mathbb{Z}/\eta,\mathbb{Z}(n+1)[2n+1])\]
by duality and \cite[Proposition 5.4]{Y1}. Apply \cite[Proposition 2.2]{Y2}.
\end{enumerate}
\end{proof}
\begin{proposition}\label{ortho}
Define
\[\begin{array}{cc}A_l=Hom_{\widetilde{DM}}(C(24\partial)(1)[1],\mathbb{Z}(l+1)[2l+2])&B_l=Hom_{\widetilde{DM}}(C(24Id_{\mathbb{Z}/\eta})(1)[1],\mathbb{Z}(l+1)[2l+2]).\end{array}\]
We have
\[\begin{array}{cc}A_l=\begin{cases}0&l<0\textrm{ or }l>1\\\mathbb{Z}/24\mathbb{Z}&l=1\\2\mathbb{Z}\oplus{W}(k)&l=0\end{cases}&B_l=\begin{cases}0&l>1\textrm{ or }l<0\\\mathbb{Z}/24\mathbb{Z}&l=1\\2\mathbb{Z}/48\mathbb{Z}&l=0\end{cases}.\end{array}\]
\end{proposition}
\begin{proof}
We use \cite[Proposition 5.4, 5.6]{Y1}. We have a long exact sequence
\[H^{2l-2,l-1}_{MW}(k,\mathbb{Z})\xrightarrow{24\partial}H^{2l,l}_{MW}(\mathbb{Z}/\eta,\mathbb{Z})\to A_l\to H^{2l-1,l-1}_{MW}(k,\mathbb{Z})\xrightarrow{24\partial}H^{2l+1,l}_{MW}(\mathbb{Z}/\eta,\mathbb{Z}).\]
If \(l<0\) or \(l>1\), the second and the fourth term vanish. Hence \(A_l=0\). If \(l=1\), the fourth term is zero and the first arrow is \(24rk:GW(k)\to\mathbb{Z}\). So \(A_l=\mathbb{Z}/24\mathbb{Z}\). Suppose \(l=0\). We have a commutative diagram with rows being distinguished triangles
\[
	\xymatrix
	{
		\mathbb{Z}/\eta\ar[r]^{24\partial}\ar[d]_{24Id}	&\mathbb{Z}(1)[2]\ar[r]\ar@{=}[d]	&C(24\partial)\ar[r]\ar[d]			&\\
		\mathbb{Z}/\eta\ar[r]^{\partial}						&\mathbb{Z}(1)[2]\ar[r]					&\mathbb{Z}[1]\ar[r]	&
	},
\]
which induces the commutative diagram with exact rows by applying \(Hom(-,\mathbb{Z}[1])\)
\[
	\xymatrix
	{
		0\ar[r]	&\mathbb{Z}\ar[r]^-h\ar[d]_{24}	&GW(k)\ar[r]\ar[d]_u	&{W}(k)\ar[r]\ar@{=}[d]	&0\\
		0\ar[r]	&\mathbb{Z}\ar[r]^-v				&A_{0,0}\ar[r]^{p}					&{W}(k)\ar[r]					&0
	}.
\]
Now suppose \(x\in{W}(k)\) has a lift \(y\in{GW}(k)\). Then define \(\varphi(x)=u(y)-12v(rk(y))\). Hence \(\varphi(h)=0\) so \(\varphi\) is well defined over \({W}(k)\) and is a section of \(p\). Finally, the term \(2\mathbb{Z}\) comes from \(Hom(\mathbb{Z}/\eta,\mathbb{Z})=2\mathbb{Z}\).

Similarly, there is a long exact sequence
\[H^{2l,l}_{MW}(\mathbb{Z}/\eta,\mathbb{Z})\xrightarrow{24}H^{2l,l}_{MW}(\mathbb{Z}/\eta,\mathbb{Z})\to B_l\to H^{2l+1,l}_{MW}(\mathbb{Z}/\eta,\mathbb{Z})\xrightarrow{24}H^{2l+1,l}_{MW}(\mathbb{Z}/\eta,\mathbb{Z}).\]
The fourth term is zero and the second term is \(\begin{cases}\mathbb{Z}&l=1\\2\mathbb{Z}&l=0\\0&\textrm{else}\end{cases}\).
\end{proof}

We can recover the computation of \cite{LM}.
\begin{coro}\label{ring11}
We have
\[\begin{array}{cc}\widetilde{CH}^i(\overline{\mathcal{M}}_{1,1})=\begin{cases}GW(k)&i=0\\W(k)\oplus2\mathbb{Z}&i=1\\\mathbb{Z}/24\mathbb{Z}&i\geq 2,i\textrm{ even}\\2\mathbb{Z}/48\mathbb{Z}&i\geq 2,i\textrm{ odd}\end{cases}&\widetilde{CH}^i(\overline{\mathcal{M}}_{1,1},O(-1))=\begin{cases}2\mathbb{Z}&i=0\\\mathbb{Z}&i=1\\\mathbb{Z}/24\mathbb{Z}&i\geq 2,i\textrm{ odd}\\2\mathbb{Z}/48\mathbb{Z}&i\geq 2,i\textrm{ even}\end{cases}.\end{array}\]
\end{coro}
\section{The MW-motive of \(\overline{\mathcal{M}}_{1,2}\)}\label{m12}
In this section, we will complete the computation of MW-motive of \(\overline{\mathcal{M}}_{1,2}\). By \cite[Corollary 2.6]{K}, \(\overline{\mathcal{M}}_{1,2}\) is the universal curve over \(\overline{\mathcal{M}}_{1,1}\). Then we define
\[Z\subseteq \mathbb{P}^2_{2,3,0}\times V_{4,6}^{\times}=\{((x:y:z),a,b)\}\]
to be the vanishing locus of \(y^2z-(x^3+axz^2+bz^3)\). Then
\begin{equation}\label{quo}\overline{\mathcal{M}}_{1,2}=[Z/\mathbb{G}_m].\end{equation}
See also \cite[Section 2]{Ma} for example.

\begin{lemma}\label{cone}
Let \(\mathscr{C}\) be a triangulated category with \(A,B,X\in\mathscr{C}\). Let \(f:A\to B\) and \(g:X\to X\). If the map \(Hom_{\mathscr{C}}(g,T)\) is bijective for \(T=A,B,A[1],B[1]\), it is bijective for \(T=C(f)\).
\end{lemma}
\begin{proof}
Follows from the commutative diagram with exact rows and the Five Lemma
\[
	\xymatrix
	{
		Hom(X,A)\ar[r]^{f_*}\ar[d]_{g^*}	&Hom(X,B)\ar[r]\ar[d]_{g^*}	&Hom(X,C(f))\ar[r]\ar[d]_{g^*}	&Hom(X,A[1])\ar[r]^-{f[1]_*}\ar[d]_{g^*}	&Hom(X,B[1])\ar[d]_{g^*}\\
		Hom(X,A)\ar[r]^{f_*}						&Hom(X,B)\ar[r]						&Hom(X,C(f))\ar[r]						&Hom(X,A[1])\ar[r]^-{f[1]_*}				&Hom(X,B[1])\\
	}.
\]
\end{proof}
The decomposition in Theorem \ref{11} shows the following:
\begin{lemma}\label{al}
The \(\mathbb{Z}(\overline{\mathcal{M}}_{1,1})=\oplus_{i=1}^{\infty}A_i\) such that for each \(i\), we have
\[\#\{j|Hom(A_i,A_j)\neq0\}<\infty.\]
Hence we have
\[End(\oplus_{i=1}^{\infty}A_i)=\prod_{j=1}^{\infty}Hom(\oplus_{i=1}^{\infty}A_i,A_j).\]
\end{lemma}

The natural map \(p:\overline{\mathcal{M}}_{1,2}\to\overline{\mathcal{M}}_{1,1}\) has a canonical infinity section \(inf:\overline{\mathcal{M}}_{1,1}\to\overline{\mathcal{M}}_{1,2}\).
\begin{theorem}\label{12}
The composite
\[\mathbb{Z}(\overline{\mathcal{M}}_{1,2}\setminus inf)\to\mathbb{Z}(\overline{\mathcal{M}}_{1,2})\to\mathbb{Z}(\overline{\mathcal{M}}_{1,1})\]
is an isomorphism in \(\widetilde{DM}(k)\). We have
\[\mathbb{Z}(\overline{\mathcal{M}}_{1,2})=\mathbb{Z}(\overline{\mathcal{M}}_{1,1})\oplus Th_{\overline{\mathcal{M}}_{1,1}}(O(-1))\]
in \(\widetilde{DM}(k)\).
\end{theorem}
\begin{proof}
Define
\[W\subseteq V_{2,3}\times V_{4,6}=(x,y,a,b)\]
to be the vanishing locus of \(y^2-(x^3+ax+b)\) and
\[\begin{array}{cc}W_1=W\setminus(0,0,0,0)&W_2=W_1\setminus\{a=b=0\}\end{array}.\]
Then \([W_2/\mathbb{G}_m]=\overline{\mathcal{M}}_{1,2}\setminus {inf}\). The \(W\) is equivariantly isomorphic to \(V_{2,3,4}\).

We have thus an identification
\[\mathbb{Z}([W_1/\mathbb{G}_m])=\mathbb{Z}((O_{\mathbb{P}^{\infty}}(2)\oplus O_{\mathbb{P}^{\infty}}(3)\oplus O_{\mathbb{P}^{\infty}}(4))^{\times}).\]
We see by the same method as in \eqref{gysin} that \(\mathbb{Z}([W_1/\mathbb{G}_m])[1]\) is the mapping cone of the map
\begin{equation}\label{w1}
	\xymatrixcolsep{1pc}
	\xymatrix
	{
		\mathbb{Z}\oplus	&\mathbb{Z}/\eta(1)[2]\oplus	&\mathbb{Z}/\eta(3)[6]\oplus\ar[d]_{24\mathrm{Id}}	&\cdots\\
									&											&\mathbb{Z}/\eta(3)[6]\oplus						&\cdots
	}.
\end{equation}
So we have
\[\mathbb{Z}([W_1/\mathbb{G}_m])=\mathbb{Z}\oplus\mathbb{Z}/\eta(1)[2]\oplus\bigoplus_{i=1}^{\infty}C(24\mathrm{Id}_{\mathbb{Z}/\eta})(2i+1)[4i+1].\]

Let us compute \(\mathbb{Z}([W_2/\mathbb{G}_m])\). We have a Cartesian square
\[
	\xymatrix
	{
		O(1)^{\times}\ar[r]^-{i'}\ar[d]	&(O(2)\oplus O(3)\oplus O(4))^{\times}\ar[d]\\
		O(1)\ar[r]^-i						&O(2)\oplus O(3)\oplus O(4)\\
	}
\]
where \(i(\varphi)=(\varphi^2,\varphi^3,0)\). The \(i\) is a normalization map of its image and the \(i'\) is a closed immersion with \(s=i'_*\). So we have the equality
\[[W_1/\mathbb{G}_m]\setminus[W_2/\mathbb{G}_m]=O(1)^{\times},\]
thus a distinguished triangle by Lemma \ref{o1}
\[\mathbb{Z}([W_2/\mathbb{G}_m])\to\mathbb{Z}([W_1/\mathbb{G}_m])\to\mathbb{Z}(2)[4]\to\cdots[1]\]
where the second arrow is determined by a morphism \(s:\mathbb{Z}/\eta((1))\to\mathbb{Z}(2)[4]\) by Proposition \ref{ortho}. This gives an exact sequece
\[GW(k)\xrightarrow{s}\widetilde{CH}^2([W_1/\mathbb{G}_m])=\mathbb{Z}\to\widetilde{CH}^2([W_2/\mathbb{G}_m])\to 0.\]
Choose a general section \((f,g,h)\in(O(2)\oplus O(3)\oplus O(4))^{\times}\). We have
\[(f,g,h)\cap O(1)^{\times}=\{f^3=g^2,h=0\}\subseteq\mathbb{P}^{\infty}.\]
The cycle \(\{f^3=g^2,h=0\}\) is linearly equivalent to \(24O(1)^2\in CH^2(\mathbb{P}^{\infty})=\widetilde{CH}^2(\mathbb{P}^{\infty})\), hence \(s=24\). This implies that
\[\mathbb{Z}([W_2/\mathbb{G}_m])=\mathbb{Z}\oplus C(24\partial)(1)[1]\oplus\bigoplus_{i=1}^{\infty}C(24\mathrm{Id}_{\mathbb{Z}/\eta})(2i+1)[4i+1].\]

Define \(D(y)=\{((x:y:z),a,b)\in Z|y\neq 0\}\). Consider the inclusions
\[[V_{4,6}^{\times}/\mathbb{G}_m]\xrightarrow{inf}[D(y)/\mathbb{G}_m]\xrightarrow{v}[(V_{-1,-3}\times V_{4,6}^{\times})/\mathbb{G}_m].\]
Now we identify the normal bundle \(N_{inf}\) of \({inf}\). The \(v\circ inf\) is the zero section
\[(O(4)\oplus O(6))^{\times}\subseteq(O(-1)\oplus O(-3))\times_{\mathbb{P}^{\infty}}(O(4)\oplus O(6))^{\times},\]
whose normal bundle is \(O(-1)\oplus O(-3)\). The \(v\) is the inclusion of the vanishing locus of
\[z=x^3+axz^2+bz^3,\]
which has degree \(-3\). We have an exact sequnce
\[0\to N_{inf}\to N_{v\circ inf}\to N_v|_{[V_{4,6}^{\times}/\mathbb{G}_m]}\to0.\]
By
\[Pic(\mathbb{P}^{\infty})=Pic([V_{4,6}^{\times}/\mathbb{G}_m])=Pic(O(-1)\oplus O(-3)\oplus O(4)\oplus O(6))=Pic([D(y)/\mathbb{G}_m]),\]
the \(N_v|_{[V_{4,6}^{\times}/\mathbb{G}_m]}\) is identified with \(O(-3)\), hence we see that \(N_{inf}=O(-1)\). Hence we have a Gysin triangle
\begin{equation}\label{gy}[W_2/\mathbb{G}_m]\xrightarrow{i}\overline{\mathcal{M}}_{1,2}\to Th_{\overline{\mathcal{M}}_{1,1}}(O(-1))\to\cdots[1].\end{equation}
Now we want to show that the composite
\[[W_2/\mathbb{G}_m]\xrightarrow{i}\overline{\mathcal{M}}_{1,2}\xrightarrow{p}\overline{\mathcal{M}}_{1,1}\]
is an isomorphism in \(\widetilde{DM}(k)\) so \eqref{gy} splits. For this, it suffices show \([p\circ i,A]_{MW}\) is an isomorphism for
\[A=\mathbb{Z},\mathbb{Z}(2)[3], \mathbb{Z}(2)[4],\mathbb{Z}/\eta(n)[2n],\mathbb{Z}/\eta(n)[2n-1]\]
where \(n>0\) is odd by Lemma \ref{cone}. We have a morphism between Gysin triangles
\begin{equation}\label{g}
	\xymatrix
	{
		[W_2/\mathbb{G}_m]\ar[r]\ar[d]_{p\circ i}		&[W_1/\mathbb{G}_m]\ar[r]\ar[d]_a						&\mathbb{Z}(2)[4]\ar[r]\ar[d]	&\cdots[1]\\
		\overline{\mathcal{M}}_{1,1}\ar[r]	&\mathbb{P}^{\infty}\ar[r]^{e(O(4)\oplus O(6))}	&\mathbb{P}^{\infty}(2)[4]\ar[r]		&\cdots[1]
	}.
\end{equation}
\begin{enumerate}
\item\(A=\mathbb{Z}\). Clear.
\item\(A=\mathbb{Z}(2)[4]\). 
Applying \(\widetilde{CH}^2(-)\) to \eqref{g}, we have a commutative diagram with exact rows
\[
	\xymatrix
	{
		GW(k)\ar[r]\ar@{=}[d]	&\mathbb{Z}\ar[r]\ar@{=}[d]	&\widetilde{CH}^2(\overline{\mathcal{M}}_{1,1})\ar[r]\ar[d]_{(p\circ i)^*}	&0\\
		GW(k)\ar[r]				&\mathbb{Z}\ar[r]					&\widetilde{CH}^2([W_2/\mathbb{G}_m])\ar[r]												&0\\
	}.
\]
So the \((p\circ i)^*\) an isomorphism.
\item\(A=\mathbb{Z}(2)[3]\). By Lemma \ref{-1} and duality we have
\[Hom(\mathbb{Z}/\eta(2i+1)[4i+2],\mathbb{Z}(2)[3])=\begin{cases}0&i>0\\2K_1^M(k)&i=0\end{cases}.\]
Applying \(H^{3,2}_{MW}(-,\mathbb{Z})\) to \eqref{g}, we have a commutative diagram with exact rows
\[
	\xymatrix
	{
		0\ar[r]	&2K_1^M(k)\ar[r]\ar[d]_{a^*}	&H^{3,2}_{MW}(\overline{\mathcal{M}}_{1,1},\mathbb{Z})\ar[r]\ar[d]_{(p\circ i)^*}	&GW(k)\ar[r]\ar@{=}[d]	&\mathbb{Z}\ar@{=}[d]\\
		0\ar[r]	&2K_1^M(k)\ar[r]			&H^{3,2}_{MW}([W_2/\mathbb{G}_m],\mathbb{Z})\ar[r]													&GW(k)\ar[r]					&\mathbb{Z}
	}.
\]
The \(a^*\) is an isomorphism by applying \(H^{3,2}_{MW}(-,\mathbb{Z})\) to the Gysin triangle
\[(O(2)\oplus O(3)\oplus O(4))^{\times}\to\mathbb{P}^{\infty}\to Th(O(1))(2)[4]\to\cdots[1].\]
So the \((p\circ i)^*\) an isomorphism.
\item\(A=\mathbb{Z}/\eta(n)[2n]\). Applying \({E_1}^{2n,n}(-)\) to \eqref{g} we get a diagram with rows being exact 
\[
	\xymatrix
	{
		{E}_1^{2n-4,n-2}(\mathbb{P}^{\infty})\ar[r]^{e(O(4)\oplus O(6))}\ar@{->>}[d]	&E_1^{2n,n}(\mathbb{P}^{\infty})\ar[r]\ar[d]_{a^*}	&E_1^{2n,n}(\overline{\mathcal{M}}_{1,1})\ar[r]\ar[d]_{(p\circ i)^*}	&0\\
		{E}_1^{2n-4,n-2}(k)\ar[r]																		&E_1^{2n,n}([W_1/\mathbb{G}_m])\ar[r]				&E_1^{2n,n}([W_2/\mathbb{G}_m])\ar[r]																		&0
	}.
\]
There is an exact sequence
\[{E}_1^{2n-4,n-2}(Th(O(1)))\xrightarrow{e(O(2)\oplus O(3)\oplus O(4))}E_1^{2n,n}(\mathbb{P}^{\infty})\xrightarrow{a^*}E_1^{2n,n}([W_1/\mathbb{G}_m])\to 0\]
so \(ker(a^*)=Im(e(O(2)\oplus O(3)\oplus O(4)))\). Suppose \(n\geq3\). We have \({E}_1^{2n-4,n-2}(k)=0\) by Remark \ref{e1} and for same reason, 
\[Im(e(O(2)\oplus O(3)\oplus O(4)))=Im(e(O(4)\oplus O(6)))=\{(24a,24b)\in CH^{n}(\mathbb{P}^{\infty})\oplus CH^{n+1}(\mathbb{P}^{\infty})|a\equiv b(\textrm{mod }2)\}.\]
So \((p\circ i)^*\) is an isomorphism by Snake Lemma. If \(n=1\), we have \(Im(e(O(2)\oplus O(3)\oplus O(4)))=0\) so \(a^*\) is an isomorphism, hence so is \((p\circ i)^*\).
\item\(A=\mathbb{Z}/\eta(n)[2n-1]\). The map
\[{E}_1^{2n-4,n-2}(\mathbb{P}^{\infty})\xrightarrow{e(O(4)\oplus O(6))}E_1^{2n,n}(\mathbb{P}^{\infty})\]
is injective by Remark \ref{e1}. Applying \({E_1}^{2n-1,n}(-)\) to \eqref{g} we get a diagram with rows being exact
\[
	\xymatrix
	{
		E_1^{2n-5,n-2}(\mathbb{P}^{\infty})\ar[r]\ar@{->>}[d]_{u}	&E_1^{2n-1,n}(\mathbb{P}^{\infty})\ar[r]\ar[d]_{a^*}	&E_1^{2n-1,n}(\overline{\mathcal{M}}_{1,1})\ar[r]\ar[d]_{(p\circ i)^*}	&0\\
		E_1^{2n-5,n-2}(k)\ar[r]^-t												&E_1^{2n-1,n}([W_1/\mathbb{G}_m])\ar[r]					&E_1^{2n-1,n}([W_2/\mathbb{G}_m])\ar[r]			&0
	}.
\]
If \(n\neq 3\) the \(E_1^{2n-5,n-2}(k)=0\). If \(n=3\), by Lemma \ref{-1},
\[E_1^{1,1}(k)=E_1^{5,3}(\mathbb{Z}/\eta(1)[2])\subseteq E_1^{5,3}([W_1/\mathbb{G}_m]).\]
So \(t\) is injective. The \(Ker(u)\) amounts to compute \(A_j=Hom(\mathbb{Z}/\eta(2j-1)[4j-2],\mathbb{Z}/\eta(n-2)[2n-5])\) since \(u\) is induced by the pullback along a rational point of \(\mathbb{P}^{\infty}\). We have by Lemma \ref{-1}
\[A_j=\begin{cases}2K_1^M(k)&j=\frac{n-1}{2}\\K_1^M(k)&j=\frac{n-3}{2}\\0&\textrm{else}\end{cases}.\]
Meanwhile, the \(Ker(a^*)\), by \eqref{w1}, amounts to study the kernel of the map
\[Hom(\mathbb{Z}/\eta(2j+1)[4i+2],\mathbb{Z}/\eta(n)[2n-1])\xrightarrow{B_j}Hom(C(24\mathrm{Id}_{\mathbb{Z}/\eta})(2j+1)[4i+1],\mathbb{Z}/\eta(n)[2n-1]),j\geq1.\]
Here the first term can be computed by Lemma \ref{-1}. For the second term we observe that the
\[Hom(\mathbb{Z}/\eta(2j+1)[4j+1],\mathbb{Z}/\eta(n)[2n-1])\]
is torsion free so we have
\[Hom(C(24\mathrm{Id}_{\mathbb{Z}/\eta})(2j+1)[4j+1],\mathbb{Z}/\eta(n)[2n-1])=Hom(\mathbb{Z}/\eta(2j+1)[4j+2],\mathbb{Z}/\eta(n)[2n-1])/24.\]
Hence we see that
\[Ker(B_j)=\begin{cases}48K_1^M(k)&j=\frac{n-1}{2}\\24K_1^M(k)&j=\frac{n-3}{2}\\0&\textrm{else}\end{cases}.\]
The map \(A_j\to Ker(B_j)\) is the multiplication by \(24\), induced by that of \(e(O(4)\oplus O(6))\), hence the map \(Ker(u)\to Ker(a^*)\) surjective. Finally the \(B_j\) is clearly surjective, hence \(a^*\) is surjective. So by Snake Lemma we conclude that \((p\circ i)^*\) is an isomorphism.

By Lemma \ref{al}, we have shown that the \(Hom(p\circ i,-)\) induces an isomorphism of \(End(\overline{\mathcal{M}}_{1,1})\). Hence the \(p\circ i\) is an isomorphism by Yoneda Lemma. So \(i\) is split injective. So we conclude the proof.
\end{enumerate}
\end{proof}
However, the result of Theorem \ref{12} cannot be extended to twisted cases.
\begin{proposition}
Denote by \(O(1)\in Pic(\overline{\mathcal{M}}_{1,2})\) the pullback of \(O(1)\in Pic(\mathbb{P}^{\infty})\). We have
\[Th_{\overline{\mathcal{M}}_{1,2}\setminus inf}(O(1))\ncong Th_{\overline{\mathcal{M}}_{1,1}}(O(1))\]
in \(\widetilde{DM}(k)\).
\end{proposition}
\begin{proof}
We have
\[Th_{[W_1/\mathbb{G}_m]}(O(1))=Th_{(O_{\mathbb{P}^{\infty}}(2)\oplus O_{\mathbb{P}^{\infty}}(3)\oplus O_{\mathbb{P}^{\infty}}(4))^{\times}}(O(1)).\]
We see by the same method as in \eqref{gysin} that \(Th_{[W_1/\mathbb{G}_m]}(O(1))[1]\) is the mapping cone of the map
\[
	\xymatrixcolsep{1pc}
	\xymatrix
	{
		\mathbb{Z}/\eta(1)[2]\oplus	&\mathbb{Z}/\eta(3)[6]\oplus\ar[d]_{24\partial}	&\mathbb{Z}/\eta(5)[10]\oplus\ar[d]_{24\mathrm{Id}}	&\cdots\\
													&\mathbb{Z}(4)[8]\oplus								&\mathbb{Z}/\eta(5)[10]\oplus										&\cdots
	}.
\]
So we have
\[Th_{[W_1/\mathbb{G}_m]}(O(1))=\mathbb{Z}/\eta(1)[2]\oplus C(24\partial)(3)[5]\oplus\bigoplus_{i=2}^{\infty}C(24\mathrm{Id}_{\mathbb{Z}/\eta})(2i+1)[4i+1].\]
We have a distinguished triangle by Lemma \ref{o1}
\[Th_{[W_2/\mathbb{G}_m]}(O(1))\to Th_{[W_1/\mathbb{G}_m]}(O(1))\to\mathbb{Z}(3)[6]\to\cdots[1],\]
where the second arrow is determined by a map \(s':C(24\partial)(3)[5]\to\mathbb{Z}(3)[6]\). Apply \(Hom(-,\mathbb{Z}(3)[6])\) to the triangle we obtain
\[GW(k)\to\widetilde{CH}^2([W_1/\mathbb{G}_m],O(1))\to\widetilde{CH}^2([W_2/\mathbb{G}_m],O(1))\to0\]
where the second term is \(W(k)\oplus2\mathbb{Z}\) by Proposition \ref{ortho}. Note that there is an identification
\[H^{*,*}_{MW}(X,\mathbb{Z})[\eta^{-1}]=H^*(X,\textbf{W})\]
by \cite[Proposition 2.17]{Y2} for \(X\in Sm/k\), where \(\textbf{W}\) is the unramified Witt sheaf. By \(\mathbb{Z}/\eta[\eta^{-1}]=0\) we have
\[\begin{array}{cc}\mathbb{Z}(O(2)\oplus O(3)\oplus O(4))[\eta^{-1}]=\mathbb{Z}[\eta^{-1}]&Th_{\overline{\mathcal{M}}_{1,1}}(O(1))[\eta^{-1}]=0\end{array}.\]
There is a commutative diagram
\[
	\xymatrix
	{
		H^0(O(1),\textbf{W})\ar[r]\ar[d]_{\cong}	&H^2(O(2)\oplus O(3)\oplus O(4),\textbf{W})=0\ar[d]\\
		H^0(O(1)^{\times},\textbf{W})\ar[r]	&H^2([W_1/\mathbb{G}_m],\textbf{W})
	}
\]
so \(s'=(0,24)\in W(k)\oplus2\mathbb{Z}\) by its realization in Chow group, where the horizontal maps are push-forwards. Hence
\[H^*(Th_{[W_2/\mathbb{G}_m]}(O(1)),\textbf{W})\neq0.\]
But by Theorem \ref{11}, we have
\[H^*(Th_{\overline{\mathcal{M}}_{1,1}}(O(1)),\textbf{W})=0\]
so we conclude that
\[Th_{[W_2/\mathbb{G}_m]}(O(1))\ncong Th_{\overline{\mathcal{M}}_{1,1}}(O(1)).\]
\end{proof}
\begin{coro}\label{ring}
Denote by \(e(O(-1))\in\widetilde{CH}^1(\overline{\mathcal{M}}_{1,1},O(-1))\) the Euler class of \(O(-1)\). We have
\[\widetilde{CH}^*(\overline{\mathcal{M}}_{1,2})=\widetilde{CH}^*(\overline{\mathcal{M}}_{1,1})\oplus\widetilde{CH}^{*-1}(\overline{\mathcal{M}}_{1,1},O(-1))\]
with multiplication given by
\[(A,B)\cdot(C,D)=(AC,AD+BC+BD\cdot e(O(-1))).\] 
\end{coro}
\begin{proof}
Denote by \(R_n=\widetilde{CH}^*(\overline{\mathcal{M}}_{1,1},O(n))\) and \(S=\widetilde{CH}^*(\overline{\mathcal{M}}_{1,2})\). By the splitting of the triangle \eqref{gy}, there is a decomposition
\[S=R_0\oplus inf_*(R_{-1})\]
as abenlian groups, where the ring map \(R_0\to S\) is given by the pullback along the projection \(p:\overline{\mathcal{M}}_{1,2}\to\overline{\mathcal{M}}_{1,1}\). Suppose \(a,b\in R_{-1}\). We have
\[p^*(a)\cdot inf_*(b)=inf_*(inf^*p^*(a)\cdot b)=inf_*(a\cdot b)\]
\[inf_*(a)\cdot inf_*(b)=inf_*(a\cdot inf^*inf_*(b))=inf_*(a\cdot b\cdot T),\]
which concludes the proof. Here the last equality follows from \cite[Theorem 3.3]{F}.
\end{proof}
Denote by \(I(k)\) the fundamental ideal of \(GW(k)\) and by \(h\in GW(k)\) the hyperbolic element. We have \(GW(k)/I(k)=\mathbb{Z}\) via the rank map, where \(h\) is sent to \(2\). Recall that the main theorem of \cite{LM} showed that the total Chow-Witt ring \(\widetilde{CH}^*(\overline{\mathcal{M}}_{1,1},O\oplus O(-1))\)
is isomorphic to
\[R=GW(k)[T,V,H]/(I(k)T,I(k)H,H^2-2h,hV,HV,V^2,24T^2,12HT^2-VT)\]
as \(\mathbb{N}\times\mathbb{Z}/2\)-graded algebra, where \(deg(T)=(1,1), deg(V)=(1,0), deg(H)=(0,1)\).
\begin{lemma}
The \(\widetilde{CH}^*(\overline{\mathcal{M}}_{1,1})\) is isomorphic to
\[R_0=\frac{GW(k)[TH,V,T^2]}{(I(k)TH,I(k)T^2,V^2,(TH)^2-2hT^2,hV,24T^2,12T^2TH-VT^2,THV)},\]
where \(deg(TH)=deg(V)=1, deg(T^2)=2\).
\end{lemma}
\begin{proof}
There is a natural map \(R_0\to R\) using the same symbols. By the equalities \((TH)^2=2hT^2,12T^2TH=VT^2,THV=0,V^2=0\), the ring can be generated by \(1,THT^{2n},TH,V,T^{2n}\) as a \(GW(k)\)-module. Then one can show that
\[(R_0)_n=\begin{cases}GW(k)&n=0\\\mathbb{Z}\cdot TH\oplus W(k)\cdot V&n=1\\\mathbb{Z}/24\cdot(T^{2})^{\frac{n}{2}}&n\geq2\textrm{ is even}\\\mathbb{Z}/24\cdot TH(T^{2})^{\frac{n-1}{2}}&n\geq2\textrm{ is odd}\end{cases}=R_{(n,0)}\]
where the last equality follows from \cite[Proposition 5.3.1]{LM}.
\end{proof}
\begin{theorem}\label{ringex}
The \(\widetilde{CH}^*(\overline{\mathcal{M}}_{1,2})\) is isomorphic to
\[\frac{GW(k)[TH,V,T^2,T',H']}{\begin{pmatrix}I(k)T',I(k)H',I(k)TH,I(k)T^2,V^2,(TH)^2-2hT^2,hV,24T^2,12T^2TH-VT^2,THV,\\H'V,T^2H'-THT',H'^2-2hT',H'^2-THH',T'^2-T^2T',T'H'-T^2H',12T^2H'-VT'\end{pmatrix}}\]
as \(\mathbb{N}\)-graded algebras, where \(deg(TH)=deg(V)=deg(H')=1, deg(T^2)=deg(T')=2\).
\end{theorem}
\begin{proof}
The algebra in the statement is rewritten as
\[S=\frac{R_0[T',H']}{\begin{pmatrix}I(k)T',I(k)H',H'V,T^2H'-THT',H'^2-2hT',\\H'^2-THH',T'^2-T^2T',T'H'-T^2H',12T^2H'-VT'\end{pmatrix}}.\]
By Corollary \ref{ring} there is an \(R_0\)-algebra homomorphism \(S\to\widetilde{CH}^*(\overline{\mathcal{M}}_{1,2})\) sending \(T'\) (resp. \(H'\)) to \(inf_*(T)\) (resp. \(inf_*(H)\)). For example, the \(T^2H'=THT'\) corresponds to
\[T^2inf_*(H)=inf_*(T\cdot TH)=inf_*(T)TH\]
and the \(T'H'-T^2H'\) corresponds to
\[inf_*(T)inf_*(H)=inf_*(T^2H)=T^2inf_*(H)\]
where equalities are in \(\widetilde{CH}^*(\overline{\mathcal{M}}_{1,2})\). By the equalities \(H'^2=2hT',T'^2=T^2T',T'H'=T^2H'\), the \(S\) is generated by \(1,T', H'\) as an \(R_0\)-module. By the equalities \(THH'=H'^2=2hT',H'V=0,T^2H'=THT',VT'=12T^2H'\), the \(S\) is generated by \(R_0,H',(T^{2})^nT',TH(T^{2})^nT'\) as a \(GW(k)\)-module. By \cite[Proposition 5.3.1]{LM}, the \(H',(T^{2})^nT',TH(T^{2})^nT'\) are the generators of \(\widetilde{CH}^*(\overline{\mathcal{M}}_{1,1},O(-1))\). So it is clear that \(S\cong\widetilde{CH}^*(\overline{\mathcal{M}}_{1,2})\) by Corollary \ref{ring}.
\end{proof}

{}
\end{document}